\documentclass[12pt]{article}
\usepackage[numbers,sort&compress]{natbib}
\usepackage{amssymb,amsmath,mathrsfs,amsfonts,amsthm}
\usepackage{enumerate}
\usepackage{paralist}
\usepackage{color}
\usepackage{epsfig}
\usepackage{graphicx}
\usepackage{setspace}
\usepackage{appendix}
\begin{document}
 \def\pd#1#2{\frac{\partial#1}{\partial#2}}
\def\dfrac{\displaystyle\frac}
\let\oldsection\section
\renewcommand\section{\setcounter{equation}{0}\oldsection}
\renewcommand\thesection{\arabic{section}}
\renewcommand\theequation{\thesection.\arabic{equation}}

\newtheorem{thm}{Theorem}[section]
\newtheorem{cor}[thm]{Corollary}
\newtheorem{lem}[thm]{Lemma}
\newtheorem{prop}[thm]{Proposition}
\newtheorem*{con}{Conjucture}
\newtheorem*{questionA}{Question}
\newtheorem*{thmA}{Theorem A}
\newtheorem*{thmB}{Theorem B}
\newtheorem{remark}{Remark}[section]
\newtheorem{definition}{Definition}[section]

\title{Effects of nonlocal dispersal strategies and heterogeneous environment on total population
\thanks{The first author is supported by Shaanxi NSF (No. S2017-ZRJJ-MS-0104) and the Fundamental Research Funds for the Central Universities (No. 310201911cx014).
The second author is  supported by  NSF  of China (No. 11971498), NSF of Guangdong Province (No. 2019A1515011339) and the Fundamental Research Funds for the Central Universities. } }

\author{Xueli Bai\\ {School of Mathematics and Statistics, Northwestern Polytechnical
University,}\\{\small 127 Youyi Road(West), Beilin 710072,
Xi'an, P. R. China.}\\
Fang Li{\thanks{Corresponding author.
E-mail: lifang55@mail.sysu.edu.cn}}\\ {School of Mathematics, Sun Yat-sen
University,}\\{\small No. 135, Xingang Xi Road, Guangzhou 510275, P. R. China.}  }

\date{}
\maketitle{}

\begin{abstract}
In this paper,  we consider the following single species model with nonlocal dispersal strategy
$$
d\mathcal {L} [\theta] (x,t)    +  \theta(x,t)   [m(x)-  \theta(x,t)]=0 ,
$$
where $\mathcal {L}$ denotes the nonlocal diffusion operator, and investigate how  the dispersal rate of the species  and the distribution of resources affect the total population.  First, we show that the upper bound for the ratio between  total population and total resource is $C\sqrt{d}$. Moreover, examples are constructed to indicate that  this upper bound is optimal.  Secondly, for a type of simplified nonlocal diffusion operator, we prove that if
$ \frac{\sup m}{\inf m}<\frac{\sqrt{5}+1}{2}$,  the total population as a function of dispersal rate $d$ admits exactly one local maximum point in  $\displaystyle  (\inf m, \sup m)$.
These results reveal essential discrepancies  between local and nonlocal dispersal strategies.
\end{abstract}

{\bf Keywords} total population, nonlocal dispersal, heterogeneity
\vskip3mm {\bf MSC (2010)}: 35K57, 92D25, 45K05

\section{Introduction}
Total population is an important indicator for  persistence of species. If the quantity is at low level, the risk of extinction will increase, while if the quantity is at high level, it will lead to  shortage of resources and intense pressure of competition, which may japodize the existing stability of the multi-species systems \cite{LouCN}. {\it An interesting problem in spatial ecology is how  the dispersal rate of the species  and the distribution of resources affect the total population.}
The main purpose of this paper  is to investigate this problem  for species adopting nonlocal dispersal strategies.

Our study  is motivated by a series of intriguing questions and work related to total equilibrium population in  a single logistic equation with random diffusion as follows
\begin{equation}\label{general single}
\begin{cases}
u_t =   d\Delta u    +   u    [m(x)-  u]  & x\in \Omega,\ t>0,\\
\frac{\partial u}{\partial \nu} =0  & x\in \partial\Omega,\ t>0,
\end{cases}
\end{equation}
where $u$ represents the population density of a species at location $x\in\Omega$  and at time $t>0$, $d$ is the  dispersal rate of the species which is assumed to be a positive constant, the habitat $\Omega$ is a  bounded domain in $\mathbb R^n$ and $\nu$ denotes the unit outward normal vector. The function $m(x)$ is the intrinsic growth rate or carrying capacity, which reflects the environmental influence on the species $u$. Unless designated otherwise, we assume that $m(x)$ satisfies the following condition:

\medskip
\noindent\textbf{(M)} \hspace{0.5cm}  $ m(x)\in L^{\infty}(\Omega),\ m(x)\geq 0 \ \textrm{and}\  m\not\equiv \textrm{const} \ \textrm{on} \  \bar\Omega.$

\medskip

It is known that if $m$ satisfies the assumption \textbf{(M)}, then for every $d>0$, the problem (\ref{general single}) admits a unique positive steady state, denoted by $\theta_{d,m}(x)$, which is globally asymptotically stable (see e.g. \cite{CC-book}). In addition, a remarkable property concerning $\theta_{d,m}(x)$ was first observed in \cite{Lou06}
\begin{equation}\label{observation}
\int_{\Omega}\theta_{d,m}(x)\,dx > \int_{\Omega}m(x)\,dx \quad\text{ for all }d > 0.
\end{equation}
Biologically, this  indicates that when coupled with diffusion, a heterogeneous environment can support a total population  larger than the  total carrying capacity of the environment, which is quite different from homogeneous environment.  Simply speaking, heterogeneity  of resources can benefit species. This theory is  further  confirmed experimentally \cite{ZDN}.

Now, define
\begin{align}\label{eqn: def of E(m)}
E(m) := \sup_{d>0}\frac{\int_{\Omega}\theta_{d,m}\,dx}{\int_{\Omega}m\,dx}.
\end{align}
The ratio $E(m)$ is a key quantity  in characterizing  global dynamics of  two-species Lotka-Volterra competition systems \cite{HN}. 
According to the observation (\ref{observation}), $E(m)>1$ for any $m$ satisfying condition (\textbf{M}). The following question was initially proposed by W.-M. Ni:
\medskip

\noindent\textbf{Question.} Is $E(m)$ bounded above independent of $m$? If so, what is the optimal bound?

\medskip
\noindent This biological question leads us to understanding how to maximize the total population  under the limited total resources by redistributing the  resources, and what would ``optimal'' distribution be, if it exists. In the one-dimensional case, i.e., when $n=1$ and $\Omega$ is an open interval, W.-M. Ni conjectured that the supremum of $E(m)$ over all $m$'s satisfying condition (\textbf{M}) is $3$. This  conjecture is confirmed in \cite{BHL}. However, for higher dimensional case, i.e., $n\geq 2$, it is proved in \cite{Jumpei} that the supremum of $E(m)$  is unbounded. Some further studies related to this question can be found in \cite{MNP, Mazari2021}. This question is also investigated in patchy environment \cite{NLouY2021}.

Another interesting issue is the dependence of the total population on its dispersal rate.
According to \cite[Theorem 1.1]{Lou06}, as either $d\rightarrow0^+$ or $d\rightarrow +\infty$, the total population  $\displaystyle \int_{\Omega}\theta_{d,m}(x)dx$ always approaches $\displaystyle \int_{\Omega} m (x)dx$. Thus together with  the observation (\ref{observation}), one sees that total population is maximized at some intermediate  diffusion rate. While in \cite{LiangLou12}, some examples are constructed to demonstrate that the total population, as a function of the random diffusion rate, can have at least two local maxima. It is  shown in \cite{Lou06} that in the competition models,  the invasion of exotic species in spatially heterogeneous habitats is closely related with the total population of the resident species at equilibrium. Hence as a result of  the complicated dependence of the total population on its dispersal rate, the invasion of exotic species
depends on the dispersal rate of the resident species in complicated manners
as well \cite{LiangLou12}.

The above discussion is about single species model with random diffusion, which is the most basic local dispersal strategy. However, in ecology, in many situations (e.g. \cite{Cain,Clark1,Clark2,schurr}), dispersal is better described  as a long range process rather than as a local one, and integral operators appear as a natural choice. A commonly used form that integrates such long range dispersal is the following nonlocal diffusion operator: 
 \begin{equation}\label{nonlocal operator}
   \mathcal{L}u :=\int_\Omega k(x,y)u(y)dy-a(x)u(x),
 \end{equation}
 where the dispersal kernel $k(x,y)\geq 0$ describes the probability to jump from one
location to another and
\begin{itemize}
\item either $a(x)=1$, which corresponds to nonlocal homogeneous Dirichlet boundary condition,
\item or $\displaystyle a(x) = \int_{\Omega} k(y,x)dy$, which corresponds to nonlocal homogeneous Neumann boundary condition.
\end{itemize}
See \cite{Rossi} for details.  This nonlocal diffusion operator appears commonly in different types of models in ecology. See \cite{Allen1996, HMMV, Kot1996, Lee2001, Lutscher, Medlock2003, Meysman2003, Mogilner-E, Othmer} and the references therein.

Studying different types of dispersal strategies in heterogeneous environments has been one of the key approaches to understand growth and survival of individual populations and coexistence of species.
In this paper, we consider  the following single species model with nonlocal dispersal strategy
\begin{equation}\label{main single}
\begin{cases}
  u_t(x,t)=d\mathcal {L} [u] (x,t)    +  u(x,t)   [m(x)-  u(x,t)]   &x\in \Omega,\ t>0,\\
  u(x,0)=u_0\geq 0, &x\in \Omega,
\end{cases}
\end{equation}
where the nonlocal operator is defined as (\ref{nonlocal operator}) and  explore properties of solutions of the problem (\ref{main single}) related to  total equilibrium population.

To be more specific, in the problem (\ref{main single}), we intend to investigate the same two issues discussed above for the  model with  random diffusion  (\ref{general single})  :
\begin{itemize}
\item properties of the upper bound of $E(m)$ defined in (\ref{eqn: def of E(m)}), where $\theta_{d,m}$ denotes the positive steady state to the problem (\ref{main single}) if exists;
\item  dependence of the total population on its dispersal rate.
\end{itemize}
Indeed, not only  the joint effects of spatial heterogeneity and nonlocal dispersal strategies on  total equilibrium population will be studied, but also  some essential  discrepancy between local and nonlocal  dispersal strategies will be  demonstrated.

It is worth mentioning that, similar to the above local diffusion case,
the properties of solutions of single  species model also play an important role in determining population dynamics of  two-species Lotka-Volterra competition systems. See \cite{BL,BL2,BL3} and the references therein.


From now on,  assume that the kernel $k$ satisfies

\medskip
\noindent\textbf{(K)} \   $k(x,y)\in C(\mathbb R^n\times \mathbb R^n)$ is nonnegative and  $k(x,x)>0$ in $\mathbb R^n$. $k(x,y)$ is symmetric,

\ \ \ \ i.e., $k(x,y)=k(y,x)$. Moreover,  $\int_{\mathbb R^n} k(x,y)dy =1$.

\medskip

First of all, we prepare the existence and uniqueness result for the model (\ref{main single}) provided that $m(x) \in L^{\infty} (\Omega)$.

\begin{thm}\label{thm-existence}
Assume that $m(x) \in L^{\infty} (\Omega)$ is nonconstant and the kernel $k$ satisfies \textbf{(K)}. Define
\begin{equation}\label{mu0}
\displaystyle\mu_0=\mu_0(m) =\sup_{0\neq\psi\in L^2(\Omega)} \frac{\int_{\Omega}\left(d\mathcal {L} [\psi] (x)\psi(x) +m(x)\psi^2(x)\right)dx}{\int_{\Omega} \psi^2(x) dx}.
\end{equation}
Then  the problem (\ref{main single})
admits a unique positive steady state in $L^{\infty}(\Omega)$  if and only if $\mu_0>0$.
\end{thm}

When $m\in C(\bar\Omega)$, the existence and uniqueness of positive steady state for the model (\ref{main single}) has been studied thoroughly.   See \cite{BZh} for symmetric operators in the one dimensional case and  \cite{BL2, Coville2010}  for nonsymmetric operators. The proofs of these studies rely on the properties of nonlocal eigenvalue problems, thus the condition $m\in C(\bar\Omega)$ is required. However, to study the questions in this paper, the condition $m(x) \in L^{\infty} (\Omega)$ is necessary.
For this purpose, we employ a new approach, which depends on the application of energy functional.

Now consider
\begin{equation}\label{main single-N}
\begin{cases}
  u_t(x,t)=d\mathcal {L}_N [u] (x,t)    +  u(x,t)   [m(x)-  u(x,t)]   &x\in \Omega,\ t>0,\\
  u(x,0)=u_0\geq 0   &x\in \Omega,
\end{cases}
\end{equation}
where
\begin{equation}\label{nonlocal operator-N}
\mathcal{L}_N \phi :=\int_\Omega k(x,y)\phi(y)dy-\int_\Omega k(y,x) \phi(x) dy.
\end{equation}
The nonlocal operator $\mathcal{L}_N$ defined above is nonlocal homogeneous Neumann boundary condition. For suitably rescaled kernels, the convergence between problems with nonlocal operator $\mathcal{L}_N$ and those with homogeneous Neumann boundary conditions is verified \cite{Rossi}.

Thanks to Theorem \ref{thm-existence},  when $m(x)$ satisfies the condition \textbf{(M)}, it is easy to see that the problem (\ref{main single-N})
admits a unique positive steady state, still denoted by $\theta_{d,m}$,  in $L^{\infty}(\Omega)$.
For further discussions, we first present some basic properties about the total equilibrium population $\displaystyle \int_{\Omega}\theta_{d,m}(x)dx$ of the model (\ref{main single-N}).



\begin{prop}\label{prop}
Assume that the assumptions \textbf{(K)} and \textbf{(M)} hold for the model (\ref{main single-N}). Let $\theta_{d,m}(x)$ denote the unique positive  steady state to the problem (\ref{main single-N}).
Then the following properties hold.
\begin{itemize}
\item[(i)]  $\displaystyle  \int_{\Omega}\theta_{d,m}(x)\,dx > \int_{\Omega}m(x)\,dx$ for all $d > 0$.
\item[(ii)] $\displaystyle \lim_{d\rightarrow 0^+} \theta_{d,m} =m$ in $L^{\infty}$. Moreover, if  assume that ${\rm ess\, inf}_{x\in\Omega} m>0$, then there exist $d_0>0$  such that $\displaystyle \int_{\Omega}\theta_{d,m}(x)dx $  is increasing in $d$ in $[0, d_0]$.
\item[(iii)] $\displaystyle \lim_{d\rightarrow \infty} \theta_{d,m} ={1\over |\Omega|} \int_{\Omega} m(x) dx$ in $L^{\infty}(\Omega)$.
\end{itemize}
\end{prop}

Proposition \ref{prop} shows that, similar to the model with random diffusion (\ref{general single}),  the property (\ref{observation}) still holds  for the model with nonlocal dispersal strategy (\ref{main single-N}).

The following  theorems completely answer the  question proposed by W.-M. Ni for the single species model  with nonlocal dispersal strategies.  Our first result indicates that the order of the supremum of $ E(m)$ is at most $O(\sqrt{d})$ when the total resources are given.

\begin{thm}\label{theorem-unbounded-rate}
	Assume that $\Omega$ is a bounded domain in $\mathbb R^n$, $n\geq 1$, and $m(x)$ satisfies the condition \textbf{(M)}. Then there exists $C_0>0$, which depends on $\displaystyle \int_{\Omega}m (x)dx$ only, such that for $d\geq 1$
	\begin{equation}\label{thm-ratio-upper}
		\sup\,\{ E(m)\,|\, m \text{ satisfies condition \textup{(\textbf{M})}}\} \leq  C_0 \sqrt{d},
	\end{equation}
where $E(m)$ is defined in (\ref{eqn: def of E(m)}) and $\theta_{d,m}$ denotes the unique positive steady state to the problem (\ref{main single-N}).
\end{thm}

Moreover, examples are constructed as follows to show that  the order $O(\sqrt{d})$ is optimal under the prescribed total resources. For $\epsilon>0$, define
	\begin{equation}\label{example-m-epsilon}
		m_{\epsilon}(x) =\begin{cases}
			\displaystyle 0  &x\in \Omega\setminus \Omega_{0,\epsilon},\\
			\displaystyle {a(x_0)\over\epsilon}   &x\in \Omega_{0,\epsilon},
		\end{cases}
	\end{equation}
	where  $\Omega_{0,\epsilon} $ denotes a ball with center $x_0\in\Omega$ and radius $\displaystyle\sqrt[n]{\epsilon}$ with $\epsilon$  small  enough such that $\Omega_{0,\epsilon}  \subset \Omega$. Note that the total resources are independent of $\epsilon$, since
$$
\int_{\Omega} m_{\epsilon}(x) dx = \int_{\Omega_{0,\epsilon}} {a(x_0)\over \epsilon} dx = \omega_n   a(x_0),
$$
where $\omega_n$ denotes the volume of the unit ball in $\mathbb R^n$.

\begin{thm}\label{theorem-unbounded-example}
	Assume that $\Omega$ is a bounded domain in $\mathbb R^n$, $n\geq 1$ and $m_{\epsilon}$ is defined in (\ref{example-m-epsilon}).
	Then there exists $C_1>0$, independent of  $d$ and $m_{\epsilon} (x)$, such that
	$$
	\int_{\Omega} \theta_{d,m_{\epsilon}} (x)dx \geq C_1 \sqrt{d},
	$$
	provided that $\displaystyle \lim_{d\rightarrow +\infty}\epsilon d = \alpha \in [0,1)$, where $\theta_{d,m_{\epsilon}}$ denotes the unique positive steady state to the problem (\ref{main single-N}) with $m(x)$ replaced by $m_{\epsilon}(x)$.
\end{thm}

\begin{remark}
According to Theorem \ref{theorem-unbounded-example}, one sees that the supremum of $E(m)$ defined in (\ref{eqn: def of E(m)}) over all $m$'s satisfying condition (\textbf{M}) for the problem (\ref{main single-N}) is always unbounded. Moreover, the unboundedness of the supremum of $E(m)$  is due to the unboundedness of diffusion rate $d$ and the dimension of domains does not affect the order of the supremum  of $E(m)$. This is dramatically different from the corresponding results discussed previously for the  local model (\ref{general single}).
\end{remark}




Now we study the second question, which is about the dependence of the total population on its dispersal rate. Thanks to Proposition \ref{prop}, one sees that the total equilibrium population  of the model (\ref{main single-N}) must be maximized at some intermediate diffusion rate for symmetric nonlocal dispersal strategies and positive resources.
In particular, for a  type of simplified nonlocal dispersal operator, we provide a sufficient condition, which guarantees that the total population, as a function of the  diffusion rate, have exactly one local maximum.

To be more specific, consider  the following  nonlocal problem
\begin{equation}\label{single-simple}
\begin{cases}
  u_t(x,t)=d [\bar u (t) - u(x,t) ]   +  u(x,t)   [m(x)-  u(x,t)]   &x\in \Omega,\ t>0,\\
  u(x,0)=u_0\geq 0, &x\in \Omega,
\end{cases}
\end{equation}
where $\bar f$ represents the spatial average of any function $f$ in $\Omega$.
From the viewpoint of biology, this simplified nonlocal dispersal operator corresponds to the case that the movement distance of the species is much larger than the diameter of the habitat. This has been illustrated in \cite{LLW}. We briefly recall it here for the convenience of readers.
Consider the nonlocal dispersal operator \cite{HMMV}
\begin{equation}\label{Hutson}
\mathcal{L}u:=d\Big[\frac{1}{L}\int_{-\infty}^\infty
k \left(\frac{x-y}{L}\right)u(y)dy-u(x)\Big],
\end{equation}
where  $k$ is a non-negative symmetric function satisfying $\displaystyle\int_{-\infty}^{\infty}k(y)dy=1$,
and $k(x-y)$ represents the probability of movement
between $x$ and $y$ and the
dispersal spread  length $L$ characterizes the movement distance.
When $L$ is sufficiently small, by formal expansion,
the  operator  $\mathcal{L}$
can be written as
$$
\mathcal{L}u=\frac{1}{2} dL^2 \left(\int_{-\infty}^\infty k(z) z^2\, dz\right) u_{xx}+O(L^3).
$$
That is, the dispersal operator  $\mathcal{L}$ for small spread length  $L$
can be approximated by random diffusion operator.
However, for large spread length $L$,
as explained in \cite{LLW}, if $u$ is a periodic function with period $l>0$, then the  operator $\mathcal{L}$
can be approximated by operator $\mathcal{L}_1$, where
$$
\mathcal{L}_1u:=d\left[ \frac{1}{l}\int_0^l u(y) dy-u(x)\right],
$$
which leads to the nonlocal dispersal operator in (\ref{single-simple}). In genetic models, this nonlocal dispersal operator is introduced by T. Nagylaki \cite{N11} to represent global panmixia, which is the limiting case of long-distance migration.

Again, due to Theorem \ref{thm-existence}, when $m(x)$ satisfies the condition \textbf{(M)}, the problem (\ref{single-simple}) admits a unique positive steady state in $L^{\infty}(\Omega)$, denoted by $\theta_d(x) $.
For simplicity, set
$$
\displaystyle \bar \theta(d) = {1\over |\Omega|} \int_{\Omega} \theta_d(x) dx .
$$
We have the following result.

\begin{thm}\label{theorem-total}
 Suppose that $m(x)$ satisfies the condition \textbf{(M)} and
 \begin{equation}\label{theorem-ratio}
 \frac{\sup_{\Omega} m}{\inf_{\Omega} m}<\frac{\sqrt{5}+1}{2},
 \end{equation}
 then there
  exists $ L\in (\inf_{\Omega} m, \sup_{\Omega} m)$ such that $\bar{\theta}(d)$ is non-decreasing in
  $d$ when $d\le L$ and non-increasing in $d$ when $d>L$.
\end{thm}

In the examples constructed in \cite{LiangLou12} for the model (\ref{general single})  with random diffusion, the authors choose $m(x) =1 + \epsilon g(x)$ with $\displaystyle \int_{\Omega} g(x )dx =0$ and $\epsilon$ sufficiently small,
and show that the total population has at least two local maxima as diffusion rate varies. However, obviously $m(x) =1 + \epsilon g(x)$  satisfies (\ref{theorem-ratio}) when $\epsilon$ is sufficiently small. This observation demonstrates that, for certain distribution of resources, the total population, as a function of diffusion rate, could have essentially different  behaviors for the local model  (\ref{general single}) and the nonlocal model (\ref{single-simple}).

This paper is organized as follows. Theorem \ref{thm-existence} is   proved in Section 2. Section 3 is devoted to the proofs of Proposition \ref{prop}, Theorems  \ref{theorem-unbounded-rate} and \ref{theorem-unbounded-example}. The proof of  Theorem \ref{theorem-total} is included in Section 4. Some miscellaneous remarks are included at the end.

\section{Existence and uniqueness of positive steady state}
In this section, we establish  the existence and uniqueness of positive steady state to the  problem of (\ref{main single}) when $m(x) \in L^{\infty} (\Omega)$.

\begin{proof}[Proof of Theorem \ref{thm-existence}]
First, if   the problem (\ref{main single}) admits a positive steady state, denoted by $\theta$, in $L^{\infty}(\Omega)$, then it is easy to see that $\mu_0>0$ by choosing $\psi=\theta$.

The rest of the proof is devoted to proving the other direction.
Assume $M:=\|m\|_{L^{\infty}}$ and let $u$ be the solution of
\begin{equation}\label{single-M}
\begin{cases}
  u_t=d\mathcal {L} [u] (x,t)    +  u(x,t)   [m(x)-  u(x,t)]   &x\in \Omega,\ t>0,\\
  u(x,0)=M &x\in \Omega.
\end{cases}
\end{equation}
 Thus, $u$ is decreasing in $t$ and there exists $\theta^*\in L^\infty(\Omega)$ such that $u(x,t)\rightarrow \theta^*(x)$ pointwisely as $t\rightarrow\infty$.  Moreover, $\theta^*$ is a steady state of (\ref{single-M}).

Now we show that $\theta^*\not\equiv 0$. Suppose that it is not true, that is $u(x,t)\rightarrow 0$ pointwisely  as $t\rightarrow\infty$.
Since $\mu_0>0$, by the definition of $\mu_0$ we can choose $0\not\equiv\psi_0\in L^2$ such that
\begin{equation}\label{eigen2}
\int_{\Omega}\left(d\mathcal {L} [\psi_0] (x)\psi_0(x) +m(x)\psi_0^2(x)\right)dx\ge \frac{\mu_0}{2}\int_{\Omega} \psi_0^2 dx>0.
\end{equation}
Let $\psi_i:=\min\{\psi_0, i\}$, obviously $\psi_i\rightarrow\psi_0$ in $L^2(\Omega)$ as $i\rightarrow\infty$. Combined with (\ref{eigen2}), we can fix $i=i_0$ large enough, such that
$$
\int_{\Omega}\left(d\mathcal {L} [\psi_{i_0}] (x)\psi_{i_0}(x) +m(x)\psi_{i_0}^2(x)\right)dx\ge \frac{\mu_0}{4} \int_{\Omega} \psi_{i_0}^2 dx > 0.
$$
Set $\phi:=\varepsilon_{i_0}\psi_{i_0}$, with $\displaystyle\varepsilon_{i_0}= \frac{1}{i_0} \min\{M,\frac{\mu_0}{8}\}$. It is routine to verify that
\begin{equation}\label{eigne-ep}
\int_{\Omega}\left(d\mathcal {L} [\phi] (x)\phi(x) +m(x)\phi^2(x)\right)dx-\frac{2}{3} \int_{\Omega} \phi^3dx \ge [\frac{\mu_0}{4}-\varepsilon_{i_0} i_0]\int_{\Omega} \phi^2 dx> 0.
\end{equation}
Suppose that $v$ is the  solution of
$$
\begin{cases}
  v_t=d\mathcal {L} [v] (x,t) +(m-v)v   &x\in \Omega,t>0,\\
  v(x,0)=\phi   &x\in \Omega,
\end{cases}
$$
and define
$$
E[v](t):=\frac{1}{2}\int_{\Omega}\left( d\mathcal {L} [v] v + mv^2 \right) dx - \frac{1}{3} \int_{\Omega}  v^3 dx.
$$
By comparison principle $\phi\le M$ implies that $v\le u$. Thus, $v\rightarrow 0$ in pointwisely as $t\rightarrow\infty$, and furthermore
\begin{equation}\label{E to 0}
E[v](t)\rightarrow 0 ~{\rm as}~ t\rightarrow\infty.
\end{equation}
However, since $k(x,y)$ is symmetric, straightforward computation yields that
$$
  \frac{d}{dt}E[v](t) =\int_{\Omega} v_t^2 dx\ge 0.
$$
Together with (\ref{eigne-ep}),  one sees that $ E[v](t)$ is a increasing function with positive initial data, which contradicts to (\ref{E to 0}).

Hence $\theta^*(x)\ge 0$ is a nontrivial steady state of (\ref{main single}).
Furthermore, denote $A:=\{x\in\Omega\, |\, \theta^*(x)=0\}$.  Due to the assumption \textbf{(K)}, a contradiction can be derived easily by integrating both sides of the equation satisfied by $\theta^*$ in $A$ if $A$ has positive measure. This yields that $\theta^*>0$ a.e. in $\Omega$.

It remains to show the uniqueness  of positive steady state to the  problem of (\ref{main single}) in $L^{\infty}(\Omega)$.  Suppose that $\theta\in L^\infty(\Omega)$ is a positive steady state of (\ref{main single}), i.e. $\theta$ satisfies
$$
d\mathcal {L} [\theta] (x)    +  \theta(x)   [m(x)-  \theta(x)]=0,
$$
By multiplying both sides by $\theta^{p-1}$ and integrating over $\Omega$,
we have
\begin{eqnarray*}
&& \int_{\Omega} \theta^{p+1}(x)dx - \int_{\Omega} m(x) \theta^p(x) dx\\
&=& d \int_{\Omega} \theta^{p-1}(x) \left[\int_\Omega k(x,y)\theta(y)dy-a(x)\theta(x) \right]dx\\
&\leq &d \int_{\Omega} \theta^{p-1}(x) \left(\int_\Omega k(x,y)dy\right)^{p-1\over p}\left(\int_{\Omega}k(x,y)\theta^{p}(y)dy\right)^{1\over p}dx-d\int_{\Omega}a(x)\theta^p(x) dx\\
&\leq & d \left(\int_{\Omega} \int_\Omega k(x,y)dy\theta^p(x) dx\right)^{p-1\over p}\left(\int_{\Omega}\int_{\Omega}k(x,y)\theta^{p}(y)dy dx\right)^{1\over p}-d\int_{\Omega}a(x)\theta^p(x) dx\\
&=&d \int_{\Omega} \int_\Omega k(x,y)dy\theta^p(x) dx -d\int_{\Omega}a(x)\theta^p(x) dx\leq 0,
\end{eqnarray*}
since $k(x,y)$ satisfies the assumption \textbf{(K)} and either $a(x)=1$ or $\displaystyle a(x) = \int_{\Omega} k(y,x)dy$. Thus it is easy to see that
$$\|\theta\|_{L^{p+1}}\le \|m\|_{L^{p+1}},$$
which yields that
$$\|\theta\|_{L^{\infty}}\le \|m\|_{L^{\infty}}=M,$$
since $p$ is arbitrary.
Then thanks to (\ref{single-M}), it follows that $\theta(x)\le \theta^*(x)$.
Straightforward computation gives
\begin{align*}
  \int_{\Omega} (\theta^*-\theta)\theta\theta^*dx&=\int_{\Omega}  (m-\theta)\theta\theta^*dx-\int_{\Omega} (m-\theta^*)\theta^*\theta dx\\
  &=-d\int_{\Omega} \mathcal{L}[\theta]\theta^*dx+d\int_{\Omega} \mathcal{L}[\theta^*]\theta dx =0,
\end{align*}
which implies that $\theta\equiv \theta^*$.
The proof is complete.
\end{proof}

\section{Ratio between total population and resources}
This section is devoted to the proofs of Proposition \ref{prop},  Theorems  \ref{theorem-unbounded-rate} and \ref{theorem-unbounded-example}.
Thanks to Theorem \ref{thm-existence}, when $m(x)$ satisfies the condition \textbf{(M)}, the problem (\ref{main single-N}) always admits a unique positive steady state, denoted by $\theta_{d,m}$, i.e., $\theta_{d,m}$ satisfies
\begin{equation}\label{single-ss}
d\left( \int_{\Omega} k(x,y) \theta(y) dy - a(x) \theta (x)\right)   +  \theta(x)   [m(x)-  \theta(x)] =0\ \ \ x\in\Omega,
\end{equation}
where
$$
a(x) = \int_{\Omega} k(y,x) dy\leq 1.
$$

\begin{proof}[Proof of Proposition \ref{prop}]
It is routine to show that since the nonlocal operator is symmetric and $\theta_{d,m}$ is nonconstant,
\begin{equation}\label{pf-prop-int}
\int_{\Omega}\theta_{d,m}(x)\,dx - \int_{\Omega}m(x)\,dx ={d \over 2}\int_{\Omega}\int_{\Omega}k(x,y) \frac{\left(\theta_{d,m}(x)- \theta_{d,m}(y)\right)^2}{\theta_{d,m}(x)\theta_{d,m}(y)}dydx>0,
\end{equation}
i.e., (i) is proved.

According to (\ref{single-ss}), it is routine to show that   the solution $\theta_{d,m}(x)$ can be expressed as follows
\begin{eqnarray*}
\theta_{d,m} (x) =
 {1\over 2} \left[m(x) -d a(x) + \sqrt{(m(x) -d a(x))^2 + 4 d \int_{\Omega} k(x,y) \theta_{d,m}(y) dy} \right],
\end{eqnarray*}
and $\displaystyle\|\theta_{d,m}\|_{L^{\infty}}\le \|m\|_{L^{\infty}}$.
Thus it is easy to see that
\begin{equation}\label{pf-prop-0}
\lim_{d\rightarrow 0^+} \theta_{d,m} (x) = m(x) \ \ \ \textrm{in}\ L^{\infty}(\Omega).
\end{equation}
Moreover, define
$$
T_m(d) = \int_{\Omega} \theta_{d,m}(x) dx.
$$
It follows from (\ref{pf-prop-int}) and (\ref{pf-prop-0})  that
$$
T_m'(0) = {1\over 2}\int_{\Omega}\int_{\Omega } \frac{  k(x,y) \left(m(x) - m (y)\right)^2}{m(x)m (y)} dy dx,
$$
and $T_m'(0)>0$  since $m(x)$ is nonconstant. (ii) is verified.

At the end, consider the case that $d \rightarrow +\infty.$ Due to (\ref{single-ss}), it is easy to see that
\begin{equation}\label{pf-prop-infty}
\lim_{d\rightarrow +\infty} \theta_{d,m}(x) - \frac{\int_{\Omega} k(x,y) \theta_{d,m}(y) dy}{a(x)} =0\ \ \textrm{in} \ L^{\infty}(\Omega).
\end{equation}
Since $\displaystyle\frac{\int_{\Omega} k(x,y) \theta_{d,m}(y) dy}{a(x)}$ is uniformly bounded and equi-continuous for $d\geq 1$, by Arzela-Ascoli Theorem, there exist $\ell(x)\in C(\bar\Omega)$ a sequence $\{d_j \}_{j\geq 1}$  with $d_j\rightarrow +\infty$ as $j\rightarrow \infty$ such that
$$
\lim_{j\rightarrow +\infty}   \frac{\int_{\Omega} k(x,y) \theta_{d_j,m}(y) dy}{a(x)} =\ell(x)\ \ \textrm{in} \ L^{\infty}(\Omega),
$$
which, together with (\ref{pf-prop-infty}), implies that
\begin{equation}\label{pf-prop-ell}
\lim_{j\rightarrow +\infty}    \theta_{d_j,m}(x)   =\ell(x)\ \ \textrm{in} \ L^{\infty}(\Omega),
\end{equation}
and
\begin{equation}\label{pf-prop-constant}
\ell(x) - \frac{\int_{\Omega} k(x,y) \ell(y) dy}{a(x)} =0.
\end{equation}
Suppose that $\ell\in C(\bar\Omega)$ is nonconstant. Denote $B := \{ x\in\bar\Omega \ | \ \ell =\max_{\bar\Omega}\ell \}$. Since $B \neq \bar\Omega$, a contradiction can be derived at $\partial B \bigcap \Omega$. Hence, the problem (\ref{pf-prop-constant}) only has constant solutions.

Also notice that by  the problem (\ref{single-ss}), one has
$$
 \int_{\Omega}   \theta_{d,m}(x)   [m(x)-  \theta_{d,m}(x)] dx =0.
$$
Thus it follows from (\ref{pf-prop-ell}) that
$$
\ell = {1\over |\Omega|} \int_{\Omega} m(x) dx,
$$
which is a fixed number. Hence the subsequence convergence in (\ref{pf-prop-ell}) can be improved to
$$
\lim_{d\rightarrow +\infty}    \theta_{d,m}(x)   =\ell\ \ \textrm{in} \ L^{\infty}(\Omega).
$$
The proof of (iii) is complete.
\end{proof}

Now we study the supremum of $ E(m)$. First, we show that the order of the supremum of $ E(m)$ is at most $O(\sqrt{d})$ when the total resources are given.

\begin{proof}[Proof of Theorem \ref{theorem-unbounded-rate}]


Set
\begin{eqnarray*}
&&\Omega_1 = \{x\in \Omega \ | \ \theta_{d,m}(x)>K_1 d \},\  \textrm{where}\ K_1= 2 \|k \|_{L^{\infty}} |\Omega|, \\
&&\Omega_2 =  \{x\in \Omega \ | \ \theta_{d,m}(x)>K_2 d \},\  \textrm{where}\ K_2= \frac{4\left(\int_{\Omega} m(x) dx+ K_1 |\Omega|\right)\|k \|_{L^{\infty}}}{\min_{\bar\Omega} a(x)} + 2 \|k \|_{L^{\infty}} |\Omega|.
\end{eqnarray*}
First, we establish a rough estimate for $\displaystyle \int_{\Omega}\theta_{d,m}(x) dx $ as follows. For any $x\in \Omega_1$,
\begin{eqnarray}\label{pf-thm-Omega1-x}
\theta_{d,m}(x) &=& m(x) + \frac{d\left( \int_{\Omega} k(x,y) \theta_{d,m}(y) dy - a(x) \theta_{d,m} (x)\right)}{\theta_{d,m}(x)}\cr
&\leq  &  m(x) + \frac{d \|k \|_{L^{\infty}} \int_{\Omega}  \theta_{d,m}(x) dx }{K_1 d}\cr
&= &  m(x) + \frac{1 }{2 |\Omega|}\int_{\Omega}  \theta_{d,m}(x) dx,
\end{eqnarray}
which implies that
$$
\int_{\Omega_1} \theta_{d,m}(x) dx\leq  \int_{\Omega} m(x) dx + \frac{|\Omega_1| }{2|\Omega| }\int_{\Omega}  \theta_{d,m}(x) dx \leq \int_{\Omega} m(x) dx + \frac{1 }{2 }\int_{\Omega}  \theta_{d,m}(x) dx.
$$
Then
$$
\int_{\Omega} \theta_{d,m}(x) dx = \int_{\Omega_1} \theta_{d,m}(x) dx+ \int_{\Omega_1^c} \theta_{d,m}(x) dx\leq \int_{\Omega} m(x) dx + \frac{1 }{2 }\int_{\Omega}  \theta_{d,m}(x) dx + K_1 d|\Omega_1^c|.
$$
Hence for $d\geq 1$,
\begin{equation}\label{pf-thm-roughbound}
\int_{\Omega} \theta_{d,m}(x) dx  \leq 2 \int_{\Omega} m(x) dx+ 2 K_1 d|\Omega_1^c|\leq 2\left(\int_{\Omega} m(x) dx +K_1 |\Omega|\right)d.
\end{equation}

Next, we prepare an estimate for  $|\Omega_2|$ in term of $d$. Denote
$$
\tilde\Omega_2 =\left \{ x\in \Omega \ \Big|\ m(x) \geq {d\over 2} a(x) \right\}.
$$
Obviously,
$$
\int_{\Omega} m(x) dx \geq \int_{\tilde\Omega_2} m(x) dx \geq {d\over 2} \min_{\bar\Omega} a(x)  |\tilde\Omega_2|,
$$
which implies that
$$
|\tilde\Omega_2| \leq  {1\over d} {2\over \min_{\bar\Omega} a(x) }  \int_{\Omega} m(x) dx.
$$
We claim that $\Omega_2 \subseteq \tilde\Omega_2$. If the claim is true, then one has
\begin{equation}\label{pf-thm-Omega2}
|\Omega_2| \leq {1\over d} {2\over \min_{\bar\Omega} a(x) }  \int_{\Omega} m(x) dx.
\end{equation}
To prove this claim, fix any $x\in \tilde\Omega_2^c$, i.e., $\displaystyle m(x) <{d\over 2} a(x)$.  Based on the equation (\ref{single-ss}),
\begin{eqnarray*}
\theta_{d,m} (x) &=&
 {1\over 2} \left[m(x) -d a(x) + \sqrt{(m(x) -d a(x))^2 + 4 d \int_{\Omega} k(x,y) \theta_{d,m}(y) dy} \right]\\
&=&\frac{  2d \int_{\Omega} k(x,y) \theta_{d,m}(y) dy}{  -m(x) +d a(x) + \sqrt{(m(x) -d a(x))^2 + 4 d \int_{\Omega} k(x,y) \theta_{d,m}(y) dy}}\\
&\leq & \frac{  2d \int_{\Omega} k(x,y) \theta_{d,m}(y) dy}{\displaystyle d a(x)}\leq \frac{  2 \|k \|_{L^{\infty}} \int_{\Omega}  \theta_{d,m}(x) dx}{  \min_{\bar\Omega} a(x)}\\
&\leq & \frac{  4\left(\int_{\Omega} m(x) dx+ K_1 |\Omega|\right)\|k \|_{L^{\infty}}}{\min_{\bar\Omega} a(x)} d,
\end{eqnarray*}
where the last inequality is due to (\ref{pf-thm-roughbound}). Hence $\theta_{d,m}(x) < K_2 d$, i.e., $x\in \Omega_2^c$. The claim is proved and thus (\ref{pf-thm-Omega2}) is valid.

Now we are ready to improve the estimate for $\displaystyle \int_{\Omega} \theta_{d,m}(x) dx$. For $x\in\Omega_2$, since $\Omega_2 \subseteq \Omega_1$,  the estimate (\ref{pf-thm-Omega1-x}) still holds, i.e.,
$$
\theta_{d,m}(x) \leq m(x) + \frac{1 }{2 |\Omega|}\int_{\Omega}  \theta_{d,m}(x) dx.
$$
Then
\begin{eqnarray*}
\int_{\Omega_2} \theta_{d,m}(x) dx &\leq& \int_{\Omega} m(x) dx + \frac{|\Omega_2| }{2 |\Omega|}\int_{\Omega}  \theta_{d,m}(x) dx\\
 &\leq& \int_{\Omega} m(x) dx + \frac{1 }{2 }\int_{\Omega_2}  \theta_{d,m}(x) dx+ \frac{|\Omega_2| }{2 |\Omega|}\int_{\Omega_2^c}  \theta_{d,m}(x) dx,
\end{eqnarray*}
which yields that
\begin{equation}\label{pf-thm-Omega2-int}
\int_{\Omega_2} \theta_{d,m}(x) dx \leq 2\int_{\Omega} m(x) dx + \frac{|\Omega_2| }{ |\Omega|}\int_{\Omega_2^c}  \theta_{d,m}(x) dx.
\end{equation}
Moreover, we analyze the solution $\theta_{d,m}$ in $\Omega_2^c$. According to the equation (\ref{single-ss}), the estimates (\ref{pf-thm-Omega2}) and (\ref{pf-thm-Omega2-int}), one has
\begin{eqnarray*}
\int_{\Omega_2^c} \theta_{d,m}^2(x) dx  &=& \int_{\Omega_2^c}  m(x)\theta_{d,m}(x) dx + d\int_{\Omega_2^c}  \left( \int_{\Omega} k(x,y) \theta_{d,m}(y) dy - a(x) \theta_{d,m} (x)\right) dx\\
&\leq & K_2 d \int_{\Omega_2^c}  m(x) dx - d\int_{\Omega_2}  \left( \int_{\Omega} k(x,y) \theta_{d,m}(y) dy - a(x) \theta_{d,m} (x)\right) dx\\
&\leq & K_2 d\int_{\Omega} m(x) dx +  d \int_{\Omega_2} \theta_{d,m}(x) dx\\
&\leq & \left(K_2 d  + 2d + {2\over \min_{\bar\Omega} a(x) } \frac{1 }{ |\Omega|}\int_{\Omega_2^c}  \theta_{d,m}(x) dx\right)\int_{\Omega} m(x) dx\\
&\leq & (K_2 +2) d\int_{\Omega} m(x) dx + {2\over |\Omega|}\left({\int_{\Omega} m(x) dx\over \min_{\bar\Omega} a(x) }\right)^2 + {1\over 2}\int_{\Omega_2^c} \theta_{d,m}^2(x) dx .
\end{eqnarray*}
This indicates that for $d\geq 1$,
$$
\int_{\Omega_2^c} \theta_{d,m}^2(x) dx \leq 2 (K_2 +2) d\int_{\Omega} m(x) dx + {4 \over |\Omega|}\left({\int_{\Omega} m(x) dx\over \min_{\bar\Omega} a(x) }\right)^2 \leq K_3 d,
$$
where
$$
K_3 =  2 (K_2 +2) \int_{\Omega} m(x) dx  + {4 \over |\Omega|}\left({1\over \min_{\bar\Omega} a(x) }\right)^2 \left(\int_{\Omega} m(x) dx\right)^2.
$$
Therefore, together with (\ref{pf-thm-Omega2-int}), for $d\geq 1$
\begin{eqnarray*}
\int_{\Omega} \theta_{d,m}(x) dx &=& \int_{\Omega_2} \theta_{d,m}(x) dx +\int_{\Omega_2^c} \theta_{d,m}(x) dx\\
&\leq &  2\int_{\Omega} m(x) dx+ \frac{|\Omega_2| }{ |\Omega|}\int_{\Omega_2^c}  \theta_{d,m}(x) dx +\int_{\Omega_2^c} \theta_{d,m}(x) dx\\
&\leq &  2\int_{\Omega} m(x) dx + 2 \left( |\Omega_2^c| \int_{\Omega_2^c} \theta_{d,m}^2(x) dx\right)^{1\over 2}\\
&\leq &  2\left( \int_{\Omega} m(x) dx + \sqrt{K_3 |\Omega|}\right) \sqrt{d}.
\end{eqnarray*}
Set
$$C_0 = 2\left( \int_{\Omega} m(x) dx + \sqrt{K_3 |\Omega|}\right)
$$ The desired estimate follows.
\end{proof}

Next we construct examples to demonstrate that the order $O(\sqrt{d})$ is optimal under the prescribed total resources.

\begin{proof}[Proof of Theorem \ref{theorem-unbounded-example}]
W.l.o.g., always assume that $\epsilon  d \leq 1$. First of all, it is routine to show that
$$
\theta_{d,m_\epsilon}(x) \leq {a(x_0) \over \epsilon}\leq {1\over \epsilon}\ \ \textrm{in}\ \Omega,
$$
and
\begin{equation*}
\theta_{d,m_\epsilon} (x)=\begin{cases}
\displaystyle {1\over 2} \left[ -d a(x) + \sqrt{d^2 a^2(x) + 4 d \int_{\Omega} k(x,y) \theta_{d,m_\epsilon} (y) dy} \right] &x\in \Omega\setminus \Omega_{0,\epsilon},\\
\displaystyle  {1\over 2} \left[ {a(x_0)\over \epsilon} -da(x)  + \sqrt{\left( {a(x_0)\over \epsilon} -da(x) \right)^2  + 4 d \int_{\Omega} k(x,y) \theta_{d,m_\epsilon} (y) dy} \right]   &x\in \Omega_{0,\epsilon}.
\end{cases}
\end{equation*}

Moreover, thanks to Theorem \ref{theorem-unbounded-rate}, one sees that
\begin{equation}\label{pf-theta-d}
\lim_{d\rightarrow + \infty}\sup_{\epsilon>0 } \frac{ \int_{\Omega} \theta_{d,m_\epsilon}(x) dx}{ d} =0,
\end{equation}
and
\begin{equation}\label{pf-int-theta-d}
\lim_{d\rightarrow + \infty} \sup_{\epsilon>0 } \frac{ \int_{\Omega}k(x,y) \theta_{d,m_\epsilon}(y) dy}{ d} =0
\end{equation}
uniformly in $\Omega$.

For $\displaystyle x\in \Omega\setminus \Omega_{0,\epsilon}$,
\begin{eqnarray*}
\theta_{d,m_\epsilon} (x) &=&
 {1\over 2} \left[ -d a(x) + \sqrt{d^2 a^2(x) + 4 d \int_{\Omega} k(x,y) \theta_{d,m_\epsilon}(y) dy} \right]\\
 &= &   {d\over 2}a(x)\left[ - 1 + \displaystyle\sqrt{  1 + {4\over a^2(x)}\frac{\int_{\Omega}k(x,y) \theta_{d,m_\epsilon}(y) dy}{ d}} \right]\\
 &=&  \frac{\displaystyle\int_{\Omega}k(x,y) \theta_{d,m_\epsilon}(y) dy}{ a(x)} - (1+\xi)^{-3/2}a^{-3}(x){1\over d}\left(\int_{\Omega}k(x,y) \theta_{d,m_\epsilon}(y) dy \right)^2,
\end{eqnarray*}
where
$$
0<\xi (x ) \leq {4\over a^2(x)}\frac{ \int_{\Omega}k(x,y) \theta_{d,m_\epsilon}(y) dy}{ d}.
$$
This yields that
\begin{eqnarray*}
&& \int_{\Omega\setminus \Omega_{0,\epsilon}}(1+\xi)^{-3/2}a^{-2}(x){1\over d}\left(\int_{\Omega}k(x,y) \theta_{d,m_\epsilon}(y) dy \right)^2  dx\\
&=&  \int_{\Omega\setminus \Omega_{0,\epsilon}}\left( \int_{\Omega}k(x,y) \theta_{d,m_\epsilon}(y) dy\right)dx -\int_{\Omega\setminus \Omega_{0,\epsilon}} a(x) \theta_{d,m_\epsilon} (x) dx\\
&=&   \int_{\Omega\setminus \Omega_{0,\epsilon}} \left(\int_{\Omega}k(y,x) \theta_{d,m_\epsilon}(x) dx\right) dy -\int_{\Omega\setminus \Omega_{0,\epsilon}} a(x) \theta_{d,m_\epsilon} (x) dx\\
&= &   \int_{\Omega} \left(a(x ) - \int_{\Omega_{0,\epsilon}} k(y,x) dy \right) \theta_{d,m_\epsilon}(x) dx  -\int_{\Omega\setminus \Omega_{0,\epsilon}} a(x) \theta_{d,m_\epsilon} (x) dx \\
&=&  \int_{\Omega_{0,\epsilon}} a(x) \theta_{d,m_\epsilon} (x) dx -  \int_{\Omega_{0,\epsilon}} \left(\int_{\Omega} k(y,x) \theta_{d,m_\epsilon}(x) dx  \right)  dy\\
&= &  \int_{\Omega_{0,\epsilon}} {a(x) \over 2} \left[{a(x_0)\over \epsilon} -da(x) + \sqrt{\left({a(x_0)\over \epsilon} -da(x)\right)^2  + 4 d \int_{\Omega} k(x,y) \theta_{d,m_\epsilon} (y) dy} \right] dx \\
&& - d \int_{\Omega_{0,\epsilon}} \frac{\int_{\Omega} k(x,y) \theta_{d,m_\epsilon}(y) dy }{d}  dx.
\end{eqnarray*}
Thus, thanks to  (\ref{pf-int-theta-d}) and the assumption $\displaystyle \lim_{d\rightarrow +\infty}\epsilon d = \alpha \in [0,1)$, one has
\begin{equation}\label{pf-thm-limit}
\lim_{d\rightarrow +\infty} \int_{\Omega\setminus \Omega_{0,\epsilon}} (1+\xi)^{-3/2}a^{-2}(x){1\over d}\left(\int_{\Omega}k(x,y) \theta_{d,m_\epsilon}(y) dy \right)^2  dx=\omega_n ( 1-\alpha )a^2(x_0),
\end{equation}
where $\omega_n$ denotes the volume of the unit ball in $\mathbb R^n$.
Notice that $a(x)$ is strictly positive and continuous  in $\Omega$, and $\lim_{d\rightarrow +\infty}  \xi(x)=0$ uniformly in $\Omega$. Hence (\ref{pf-thm-limit}) indicates that there exists a constant $C>0$ such that for $d$ large,
\begin{equation}\label{pf-lim-(1-1/d)}
\int_{\Omega\setminus \Omega_{0,\epsilon}} {1\over d}\left(\int_{\Omega}k(x,y) \theta_{d,m_\epsilon}(y) dy \right)^2  dx \geq C.
\end{equation}

Hence under the assumption   $\displaystyle \lim_{d\rightarrow +\infty}\epsilon d = \alpha \in [0,1)$, for $d$ large, one has
$$
 Cd \leq  \int_{\Omega} \left(\int_{\Omega}k(x,y) \theta_{d,m_\epsilon}(y) dy \right)^2 dx\leq \| k\|^2_{L^{\infty}} |\Omega| \left(\int_{\Omega}  \theta_{d,m_\epsilon}(y) dy \right)^2.
$$
Therefore,
$$
 \int_{\Omega}  \theta_{d,m_\epsilon}(x) dx \geq \sqrt{{C \over 2 |\Omega|} } {1\over\| k\|_{L^{\infty}}}  \sqrt{d}.
$$
This completes the proof.
\end{proof}

\section{Effects of diffusion rate and source}
The purpose of this section is to investigate how the total population   depends on $d$
for  the model (\ref{single-simple}) with a type of simplified nonlocal dispersal operator.

\begin{proof}[Proof of Theorem \ref{theorem-total}]
Suppose that $\bar{\theta}(d)$ admits more than one local maxima as $d$ varies, then $\bar{\theta}(d)$ must have at least one local minimal point $d_0>0$, i.e., $\bar\theta^{\prime}(d_0)=0$ and $\bar\theta'' (d_0)\geq 0$.  In the rest of this proof,  we use the symbol  $^\prime$ to denote the derivative  in $d$.

First of all, based on the equation satisfied by $\theta_d$ as follows
\begin{equation*}
d\left( \bar \theta  - \theta (x)\right)   +  \theta(x)   [m(x)-  \theta(x)] =0\ \ \ x\in\Omega,
\end{equation*}
it is easy to derive that
\begin{equation}\label{pf-thm-expression}
\theta_d(x)={m(x)-d\over 2}+{1\over 2}
\sqrt{\left[m(x)-d\right]^2+4d\bar\theta_d(d)},
\end{equation}
which yields that
\begin{equation}\label{pf-thm-identity}
\sqrt{(m(x)-d)^2+4d\bar\theta_d(d)}=2\theta_d(x) -m(x)+d=\theta_d(x) +\frac{d\bar\theta_d}{\theta_d(x)}.
\mathcal{}\end{equation}
and
\begin{equation}\label{pf-thm-theta'}
\theta_d ' =\frac{\bar \theta_d-\theta_d(x) + d\bar\theta_d '}{2\theta_d(x) -m(x)+d}.
\end{equation}
Then (\ref{pf-thm-theta'}) implies that
\begin{eqnarray}\label{pf-thm-bar-theta'}
 \bar\theta_d^\prime &=& \left( 1-{1\over |\Omega|}\int_{\Omega} \frac{d}{2\theta_d(x)-m(x)+d} dx\right)^{-1}{1\over |\Omega|}\int_{\Omega} \frac{\bar\theta_d-\theta_d(x)}{2\theta_d(x)-m(x)+d} dx \cr
    &=& \left( \int_{\Omega} \frac{2\theta_d(x)-m(x)}{2\theta_d(x)-m(x)+d}dx \right)^{-1}\int_{\Omega}  \frac{\bar\theta_d-\theta_d(x)}{2\theta_d(x)-m(x)+d} dx.
\end{eqnarray}
We remark that thanks to (\ref{pf-thm-identity}),
$$
1-{1\over |\Omega|}\int_{\Omega} \frac{d}{2\theta_d(x)-m(x)+d} dx = 1 - {1\over |\Omega|}\int_{\Omega} d\left( \theta_d(x) +\frac{d\bar\theta_d}{\theta_d(x)} \right)^{-1} dx >0,
$$
thus $\bar\theta_d'$ is always well defined.

To derive the expression for $\bar\theta_d ''$, for clarity, set
$$
f: =\frac{\displaystyle\bar\theta_d-\theta_d(x)}{\displaystyle 2\theta_d(x)-m(x)+d},\ \
g: = \frac{2\theta_d(x)-m(x)}{2\theta_d(x)-m(x)+d}.
$$
Then
$$
\bar\theta_d ' = \frac{\displaystyle \int_{\Omega} f dx}{\displaystyle \int_{\Omega} g dx},\ \
\bar\theta_d ''  =\frac{\displaystyle\int_{\Omega} f^\prime dx \int_{\Omega} g dx -\int_{\Omega} f dx \int_{\Omega} g^\prime dx}{\displaystyle \left(\int_{\Omega} g dx \right)^2}.
$$
In particular, according to the assumption that  $\bar\theta_d^{\prime}(d_0)=0$ and $\bar\theta_d'' (d_0)\geq 0$, one has
$$
\int_{\Omega} f dx\Big |_{d=d_0} = \int_{\Omega} \frac{\bar\theta_d-\theta_d(x)}{2\theta_d(x)-m(x)+d} dx \Big |_{d=d_0} =0,
$$
and
$$
\bar\theta_d^{\prime\prime}(d_0)=\frac{\displaystyle\int_{\Omega}
f^\prime dx \Big|_{d=d_0} }{\displaystyle\int_{\Omega} g dx \Big|_{d=d_0}}\geq 0,
$$
i.e., $\displaystyle \int_{\Omega} f^\prime dx \Big|_{d=d_0}  \geq 0$.

Furthermore,  based on (\ref{pf-thm-identity}), (\ref{pf-thm-theta'}) and the assumption that  $\bar\theta_d^{\prime}(d_0)=0$, it is standard to compute as follows
\begin{eqnarray*}
 f^\prime  \big|_{d=d_0} &=& \frac{(\bar\theta_d-\theta_d)^\prime(2\theta_d-m+d)-(\bar\theta_d-\theta_d)
 (2\theta'_d +1)}{(2\theta_d -m+d)^2}\Big|_{d=d_0}\\
 &=& \frac{-\theta_d^\prime(2\theta_d-m+d)-2\theta'_d(\bar\theta_d-\theta_d) -(\bar\theta_d-\theta_d)}{(2\theta_d -m+d)^2}\Big|_{d=d_0}\\
 &=& \frac{-\theta_d^\prime(2\theta_d-m+d)-2\theta'_d(\bar\theta_d-\theta_d) -\theta_d^\prime(2\theta_d-m+d)}{(2\theta_d -m+d)^2}\Big|_{d=d_0}\\
  &=& \frac{-2\theta_d'}{(2\theta_d -m+d)^2} \left( \theta_d +\frac{d\bar\theta_d}{\theta_d} +\bar\theta_d-\theta_d \right) \Big|_{d=d_0}\\
 &= & -2\bar\theta_d \theta_d^\prime  \left(\theta_d +\frac{d\bar\theta_d}{\theta_d}\right)^{-2}\left( \frac{d}{\theta_d} +1 \right)\Big|_{d=d_0}.
\end{eqnarray*}
Thus
\begin{eqnarray*}
\int_{\Omega} f^\prime dx \Big|_{d=d_0} &=& -2 \bar\theta_d\int_{\Omega} \theta_d^\prime \frac{d\theta_d + \theta_d^2}{\left(d\bar\theta_d + \theta_d^2\right)^2}dx \Big|_{d=d_0}\\
&=& -2 \bar\theta_d\int_{\Omega} \theta_d^\prime \left(\frac{d\theta_d + \theta_d^2}{\left(d\bar\theta_d + \theta_d^2\right)^2}-\frac{1}{d\bar\theta_d + \bar\theta_d^2}  \right)dx \Big|_{d=d_0} \\
&=& -2 \bar\theta_d \int_{\Omega} \frac{\theta_d \left( \bar\theta_d -\theta_d \right)}{d\bar\theta_d + \theta_d^2}\frac{ \left( \bar\theta_d -\theta_d \right)\left(\theta_d^3 +\bar\theta_d \theta_d^2+ d\bar\theta_d\theta_d- d^2\bar\theta_d \right)}{\left(d\bar\theta_d + \theta_d^2\right)^2\left(d\bar\theta_d + \bar\theta_d^2\right)} dx  \Big|_{d=d_0}.
\end{eqnarray*}

Notice that according to  \cite[Theorem 1.1]{LLW}, $\bar\theta_d$ is strictly increasing when $d<{1\over 2}(\bar m+\inf_{\Omega}  m)$
and  strictly decreasing when $d>\sup_{\Omega}  m$. Hence $d_0 \in (\inf_{\Omega} m, \sup_{\Omega} m)$.
Moreover, by the equation satisfied by $\theta_d$, it is easy to show that $\theta_{d_0}(x) \in (\inf_{\Omega} m, \sup_{\Omega} m)$ in $\Omega$. The assumption (\ref{theorem-ratio}) guarantees that
$$
\frac{\theta_{d_0}(x)}{d_0}>  \frac{\inf_{\Omega} m}{\sup_{\Omega} m}>\frac{\sqrt{5}-1}{2}\ \ \textrm{in}\ \Omega,
$$
which implies that
$$
\theta_{d_0}^3 +\bar\theta_{d_0} \theta_{d_0}^2+ d_0\bar\theta_{d_0}\theta_{d_0}- d_0^2\bar\theta_{d_0} >
\bar\theta_{d_0}[\theta_{d_0}^2+d_0\theta_{d_0}-d_0^2]>0\ \ \textrm{in}\ \Omega.
$$
This contradicts to the assumption that $d_0$ is a local maximum point. Therefore, there
  exists $L>0$ such that $\bar{\theta}(d)$ is non-decreasing in
  $d$ when $d\le L$, non-increasing in $d$ when $d>L$ and the property that $\displaystyle L\in (\inf m, \sup m)$ easily follows from  \cite[Theorem 1.1]{LLW}.

The proof is complete.


\end{proof}

\section{Miscellaneous remarks}

Logistic equation, introduced by  Verhulst in 1838 \cite{Verhulst}, is one of the most fundamental models in population dynamics. In 1951, random diffusion was introduced to model dispersal
behavior of a population \cite{Skellam}. Since then, reaction-diffusion models, which incorporate dispersal strategies, growth rates and carrying capacities, provide a good framework for studying questions in ecology.  There are tremendous studies in this direction, see the books \cite{CC-book, Ni-book,OL-book}. 

It is known that the logistic model with random diffusion (\ref{general single}) and that  with nonlocal dispersal (\ref{main single}) share a lot of similarity in qualitative properties of solutions except for regularity.
However, in this paper,  for two specific issues related to total population, our results show serious discrepancies between  the local model (\ref{general single}) and the nonlocal one (\ref{main single}).

First, for the local model (\ref{general single}), in the one-dimensional case, the ratio of the total population  to  the total resources is always less than 3 \cite{BHL}. However, for the nonlocal model (\ref{main single-N}), the supremum of this ratio is always unbounded regardless of dimension of domains. Indeed, our results indicate that for certain distribution of resources, this ratio goes to infinity with order $\sqrt{d}$ as the dispersal rate $d\rightarrow\infty$.

The second issue is about the dependence of the total population on its dispersal rate. When the distribution of resources is a perturbation of a constant, the total population, as a function of $d$, in the local model (\ref{general single}) could have multiple maximum points \cite{LiangLou12}, while in the nonlocal model (\ref{single-simple}), which is a special case of the model (\ref{main single}), it admits exactly one maximum point.

Based on the biological and mathematical meanings of these issues discussed before,
these discrepancies reflect  essential differences in local and nonlocal dispersal strategies from some concrete and subtle aspects. Our exploration in this direction is quite preliminary and more related problems are waiting to be investigated. For example,  in the first issue, how to maximize total population under the limited total resources  in  nonlocal models? In the second issue, it is still unknown  whether it is possible for the total population to have  multiple maximum points in the model (\ref{single-simple}). Moreover, the results obtained  are about the model (\ref{single-simple}) with simplified nonlocal operators and barely anything is known for more general nonlocal operators.
We will return to these problems in a future paper.

\end{document}